\documentclass[12pt,a4paper]{article}
\usepackage{amsfonts, latexsym, amsthm, color}

\def\seq#1#2#3{#1_{#2},\,\ldots,#1_{#3}}
\def\w{\widetilde}

\def\vv{{\underline{v}}}
\def\tt{{\underline{t}}}

\def\mm{\underline{m}}

\def\1{\underline{1}}

\def\Z{\Bbb Z}

\def\C{\Bbb C}

\def\OO{{\cal O}}
\def\X{{\cal X}}
\def\D{{\cal D}}

\def\CP{\Bbb C\Bbb P}

\newtheorem{theorem}{Theorem}

\newtheorem{lemma}{Lemma}

\newenvironment{remark}
{\smallskip\noindent{\bf Remark\/}.}{\smallskip\par}

\title{The Poincar\'e series of divisorial valuations in the
plane defines the topology of the
set of divisors \footnote{Math. Subject Class. 14H20,
32S45. Keywords: Poincar\'e series, divisorial valuations,
plane curve singularities.}
}
\author{
A.~Campillo
\and F.~Delgado
\thanks{First two authors were partially supported
by the grant MTM2007-64704. Address: University of Valladolid,
Dept.
of Algebra, Geometry and Topology, 47011 Valladolid, Spain. E-mail:
campillo\symbol{'100}agt.uva.es, fdelgado\symbol{'100}agt.uva.es}
\and S.M.~Gusein-Zade \thanks{The research was partially
supported by the grants
RFBR-007-00593, INTAS-05-7805 and  NWO-RFBR 047.011.2004.026.
Address: Moscow State University, Faculty of
Mathematics and Mechanics, Moscow, GSP-1, 119991, Russia. E-mail:
sabir\symbol{'100}mccme.ru}
}

\date{}
\begin{document}
\sloppy
\def\eps{\varepsilon}

\maketitle

\medskip
In \cite{yamamoto}, it was proved that the Alexander polynomial in
several variables of a (reducible) plane curve singularity (i.e.,
of the corresponding link) defines the topology of the curve
singularity and therefore its minimal embedded resolution. To a
plane curve singularity one associates a multi-index filtration on
the ring ${\cal O}_{\C^2,0}$ of germs of functions of two
variables defined by the orders of a function on
irreducible
components of the curve. In \cite{duke}, there was computed the
Poincar\'e series of this filtration which turned out to coincide
with the Alexander polynomial of the curve.
For a finite set of divisorial valuations on the ring ${\cal
O}_{\C^2,0}$ corresponding to some components of the exceptional
divisor of a modification of $(\C^2,0)$, in \cite{divisorial},
there was obtained a formula for the Poincar\'e series of the
corresponding multi-index filtration similar to the one from
\cite{duke}. Here we show that the Poincar\'e series of a set of
divisorial valuations on the ring ${\cal O}_{\C^2,0}$ defines
``the topology of the set of the divisors" in the sense that it
defines the minimal resolution of this set up to combinatorial
equivalence. In \cite{veys}, there was defined a notion of the
zeta-function of an ideal. This notion can be adapted to finite
sets of ideals giving a notion of the ``Alexander polynomial" (in
several variables) of a set of ideals (a mixture of notions
introduced in \cite{sabbah} and \cite{veys}). To a divisorial
valuation on the ring ${\cal O}_{\C^2,0}$ there corresponds a
natural ideal in the ring ${\cal O}_{\C^2,0}$. This ideal is
generated by equations of irreducible curves whose strict
transforms on the space of the modification intersect the
corresponding divisor. One can show that the Poincar\'e series of
a set of divisorial valuations on the ring ${\cal O}_{\C^2,0}$
coincides with the Alexander polynomial of the corresponding set
of ideals. In these terms, one can say that the Alexander
polynomial of a set of divisorial valuations on the ring ${\cal
O}_{\C^2,0}$ defines the topology of the set of divisors. If one
takes {\it\bf all} the components of the exceptional divisor of a
modification then the Poincar\'e series defines the topology of
the modification also for normal surface singularities:
\cite{cutkosky}. Notice that in this case the Poincar\'e series is
not, generally speaking, a topological invariant.

We also
give
a proof of the statement for curves somewhat simpler than the one
in \cite{yamamoto}.

Let $\pi: ({\cal X}, {\cal D})\to (\C^2,0)$ be a modification of
the complex plane $\C^2$, i.e., a proper analytic map of a
nonsingular surface ${\cal X}$ which is an isomorphism outside of
the origin in $\C^2$ and such that ${\cal D}= \pi^{-1}(0)$ is a
normal crossing divisor on ${\cal X}$. The modification $\pi$ is
obtained by a sequence of point blowing-ups. The exceptional
divisor ${\cal D}$ is the union of irreducible components
$E_{\sigma}$ ($\sigma\in \Gamma$), each of them is isomorphic to
the complex projective line $\CP^1$. The dual graph of the
modification $\pi: ({\cal X}, {\cal D})\to (\C^2,0)$ is the graph
whose vertices correspond to irreducible components $E_{\sigma}$
of the exceptional divisor ${\cal D}$ (i.e. to elements of the set
$\Gamma$), two vertices are connected by an edge iff the
corresponding components intersect (at a point). The dual graph of
the modification $\pi$ is a tree. The set $\Gamma$ of vertices of
the dual graph inherits a partial order defined by representation
of the modification as a sequence of blowing-ups: a component
$E_{\sigma'}$
is ``greater" than another component $E_{\sigma}$ if the
exceptional divisor of the minimal modification which contains
$E_{\sigma'}$ also contains $E_{\sigma}$ ($\sigma'>\sigma$). Two
modifications of the plane are {\it combinatorially equivalent\/}
if their dual graphs together with the partial orders of vertices
are isomorphic.

Let $E_{\sigma}$, $\sigma\in \Gamma$,  be a component of the
exceptional divisor ${\cal D}$. For a function germ $f$ from the
ring ${\cal O}_{\C^2,0}$ of germs of functions of two variables,
let $v_\sigma(f)$ be the multiplicity of the lifting $f\circ \pi$
of the function $f$ to the space $\cal X$ of the modification
along the component $E_{\sigma}$ ($v_{\sigma}(0):=\infty$). The
order function $v_{\sigma}: {\cal O}_{\C^2,0}  \to \Z_{\ge 0}\cup
\{\infty\}$ defines a valuation on the field of quotients of the
ring ${\cal O}_{\C^2,0}$: the divisorial valuation defined by the
component $E_{\sigma}$. (A function $v: \OO_{\C^2,0}\to \Z_{\ge
0}\cup \{\infty\}$ is called {\it an order function\/} if
$v(\lambda f) = v(f)$ for $\lambda\neq 0$ and $v(f_1+f_2)\ge \min
\{v(f_1), v(f_2)\}$.) Let $\stackrel{\bullet}{E_\sigma}$ be the
smooth part of the component $E_\sigma$ in the exceptional divisor
$\D$, i.e. $E_\sigma$ itself without intersection points with
other components of $\D$. Let $\w{L_\sigma}$ be a germ of a smooth
curve on the space $\X$ of the modification transversal to the
component $E_\sigma$ at a smooth point of $\D$, i.e. at a point of
$\stackrel{\bullet}{E_\sigma}$. The image $L_\sigma =
\pi(\w{L_\sigma})\subset (\C^2,0)$ is called {\it a curvette\/}
corresponding to the component $E_\sigma$.

Let us fix $r$ different components $E_{1}, \ldots, E_{r}$ of the
exceptional divisor ${\cal D}$ ($\{1,\ldots,r\}\subset \Gamma$),
let $\vv :=(\seq v1r)\in\Z^r$ and $\vv(f)
:=(v_1(f),\ldots,v_r(f))$ ($f\in {\cal O}_{\C^2,0}$). The map
$\vv: {\cal O}_{\C^2,0} \to (\Z\cup\{\infty\})^r$ defines a
multi-index filtration (by ideals) on the ring ${\cal
O}_{\C^2,0}$: for $\vv\in \Z^r$, the corresponding ideal is
$J(\vv) = \{f\in {\cal O}_{\C^2,0}: \; \vv(f)\ge \vv\}$. (Here
$(\seq v1r)\geq (\seq{v'}1r)$ if and only if $v_i\geq v'_i$ for
all $i = 1,\ldots,r$. It is sufficient to define the ideal
$J(\vv)$ only for nonnegative $\vv$, i.e. for $\vv \in
\Z_{\ge0}^r$. However, for the definition below, it is convenient
to assume that $\vv\in\Z^r$.)

Let $L(\seq t1r) := \sum\limits_{\vv\in\Z^r} \dim\left(
J(\vv)/J(\vv+\1)\right)\cdot\tt^{\,\vv}$ be a Laurent series in
the
variables $t_1$, \dots, $t_r$ (generally speaking, infinite in all
directions; here $\1 = (1,\ldots, 1)$). One can see that along each line in the
lattice $\Z^r$ parallel to a coordinate one the coefficient at
$\tt^{\,\vv}$
is the same for $\vv$ from the nonpositive part of the line. This
implies that
$$
P'(\seq t1r) = L(\seq t1r)\cdot \prod_{i=1}^r (t_i-1)
$$
is a power series in the variables $\seq t1r$, i.e. an element of
$\Z[[\seq t1r]]$. The series
$$
P_{\{v_i\}} (\seq t1r) = \frac{P'(\seq t1r)}{t_1\cdot\ldots\cdot
t_r -1}
$$
is called {\it the Poincar\'e series\/} of the collection
$\{v_i\}$ of divisorial valuations.

The minimal resolution of a set of divisorial valuations is the
minimal modification in which all the divisors appear.

\begin{theorem}\label{theo1}
The Poincar\'e series of a set of divisorial valuations on the
ring ${\cal O}_{\C^2,0}$ defines the minimal resolution of the set
of divisors up to combinatorial equivalence.
\end{theorem}

\begin{proof}
Let $(m_{\sigma \delta})$ be the inverse of minus the intersection
matrix $(E_\sigma \circ E_\delta)$ of components of the
exceptional divisor $\D$. For $\sigma\neq \delta$ the intersection
number $(E_\sigma\circ E_\delta)$  is equal to $1$ if the
components $E_\sigma$ and $E_\delta$ intersect (at a point) and is
equal to $0$ otherwise. The self-intersection number
$(E_\sigma\circ E_\sigma)$ is a negative integer. The numbers
$m_{\sigma \delta}$ are positive integers and $\det(m_{\sigma
\delta})=1$. The number $m_{\sigma \delta}$ is also equal to
$v_{\delta}(h_{\sigma}) = v_{\sigma}(h_\delta)$, where
$h_{\sigma}=0$ is an equation of a curvette $L_\sigma$
corresponding to the component $E_\sigma$ and is equal to the
intersection number $(L_\sigma\circ L_\delta)$ of curvettes
corresponding to the components $E_\sigma$ and $E_\delta$. For
$\sigma\in\Gamma$, let $\mm_{\,\sigma}:= (m_{\sigma 1}, \ldots,
m_{\sigma r})\in \Z^r_{\ge 0}$. Let $\chi(Z)$ be the Euler
characteristic of the space $Z$. According to \cite{divisorial}
one has:
\begin{equation}\label{eqPS}
P_{\{v_i\}} (\seq t1r) =
\prod\limits_{\sigma\in \Gamma} (1-
\tt^{\,\mm_\sigma})^{-\chi(\stackrel{\bullet}{E_\sigma})}\; .
\end{equation}

The Poincar\'e series $P_{\{v_i\}}(\seq t1r)$ as a power series in
$\seq t1r$ itself defines the factorization of the form
$\prod\limits_{\mm\in\Z_{\ge 0}^r\setminus\{0\}}(1-
\tt^{\,\mm})^{k_{\mm}}$. Moreover the equation~(\ref{eqPS})
implies {\it the projection formula\/}: if $\{\seq
i1{\ell}\}\subset\{1,\ldots,r\}$, then the Poincar\'e series of
the $\ell$-index filtration corresponding to the divisorial
valuations $v_{i_1},\ldots, v_{i_{\ell}}$ is obtained from the
Poincar\'e series $P_{\{v_i\}}(\seq t1r)$ by substituting the
variables $t_i$ with $i\notin \{\seq i1{\ell}\}$ by $1$.

\begin{remark}
The last property does not hold for the Poincar\'e series of the
filtration defined by orders of a function germ on irreducible
components of a plane curve singularity (see \cite{duke}). This
makes the proof of the corresponding statement for curves
(Theorem~2 below) somewhat different.
\end{remark}

\begin{figure}[h]
$$
\unitlength=1.00mm
\begin{picture}(80.00,20.00)(-10,13)
\thicklines \put(-5,30){\line(1,0){41}}
\put(44,30){\line(1,0){31}} \put(38,30){\circle*{0.5}}
\put(40,30){\circle*{0.5}} \put(42,30){\circle*{0.5}}
\put(30,10){\line(0,1){20}} \put(50,20){\line(0,1){10}}
\put(60,10){\line(0,1){20}} \put(10,15){\line(0,1){15}}
\thinlines
\put(20,30){\circle*{1}} \put(30,30){\circle*{1}}
\put(50,30){\circle*{1}} \put(60,30){\circle*{1}}
\put(65,30){\circle*{1}} \put(70,30){\circle*{1}}
\put(75,30){\circle*{1}} \put(75,30){\circle{2}}

\put(30,20){\circle*{1}}
\put(60,25){\circle*{1}}
\put(60,20){\circle*{1}}
\put(60,15){\circle*{1}}
\put(10,30){\circle*{1}}
\put(30,10){\circle*{1}} \put(50,20){\circle*{1}}
\put(60,10){\circle*{1}}
\put(-5,30){\circle*{1}}
\put(0,30){\circle*{1}} \put(5,30){\circle*{1}}
\put(15,30){\circle*{1}} \put(25,30){\circle*{1}}
\put(35,30){\circle*{1}} \put(45,30){\circle*{1}}
\put(55,30){\circle*{1}} \put(10,25){\circle*{1}}
\put(10,20){\circle*{1}} \put(10,15){\circle*{1}}
\put(30,25){\circle*{1}} \put(30,15){\circle*{1}}
\put(35,30){\circle*{1}} \put(-9,25){{\bf 1}=$\sigma_0$}
\put(11.5,14){$\sigma_1$}
\put(31.5,9){$\sigma_2$}
\put(61.5,9){$\sigma_g$}
\put(9,32){$\tau_1$} \put(29,32){$\tau_2$}
\put(57.5,33){$\tau_g$}
\put(67.5,33){$c$}
\put(77,27){$\sigma_{g+1}$}
\end{picture}
$$
\caption{The dual graph of a divisorial valuation.} \label{fig1}
\end{figure}
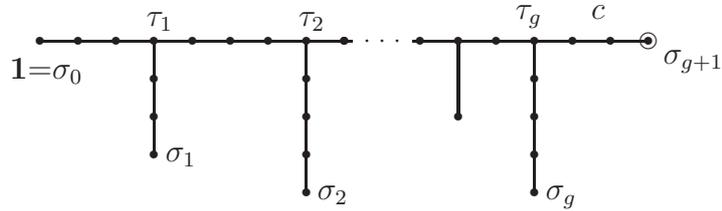

The dual graph of the minimal resolution of one divisorial
valuation on the ring $\OO_{\C^2,0}$ has the form shown on
Figure~{\ref{fig1}}. Here the corresponding component of the
exceptional divisor is marked by the circle {$\circ$}, $g$ is the
number of Puiseux pairs of a curvette corresponding to the
divisor, and the length $c$ of the ``last tail" may be equal to
zero ($\sigma_{g+1}=\tau_g$). The Poincar\'e series of this
divisorial filtration is equal to
\begin{equation}\label{eq2}
P(t) = \frac{\prod\limits_{i=1}^{\ell-1} (1-t^{m_{\tau_i}})}
{\prod\limits_{i=0}^{\ell} (1-t^{m_{\sigma_i}})
}
\end{equation}
where $m_\sigma:=m_{\sigma\sigma_{g+1}}(=v(h_\sigma))$,
$m_{\sigma_0}< m_{\sigma_1} <\ldots < m_{\sigma_{\ell}}$,
$m_{\tau_1}< m_{\tau_2} <\ldots < m_{\tau_{\ell-1}}$, and either
$\ell = g+1$ (if the length $c$ of the last tail is positive) or
$\ell=g$ (otherwise). Moreover, there is no cancellation in the
right hand side of the equation (\ref{eq2}) (since $m_{\tau_i}\neq
m_{\sigma_j}$ except possible coincidence of $\tau_g$ and
$\sigma_{g+1}$), i.e. all factors participate in (\ref{eq2})
explicitly. One can see that the second situation ($\ell=g$) takes
place if and only if $m_{\sigma_\ell}$ does not belong to the
semigroup generated by $m_{\sigma_0},\ldots, m_{\sigma_{\ell-1}}$.
In this case $m_{\sigma_0},\ldots, m_{\sigma_{\ell}}$ is the
minimal set of generators of the semigroup of values of the
divisorial valuation (see e.g. \cite{semigroup}). This  is
equivalent to the condition $e_{\ell-1}:=\gcd (m_{\sigma_0},
\ldots, m_{\sigma_{\ell-1}})
>1$.
If $\ell=g+1$, then the minimal set of generators of the semigroup
is $m_{\sigma_0},\ldots, m_{\sigma_{\ell-1}}$. Moreover,
$c=m_{\sigma_{\ell}}- m_{\tau_{\ell-1}}$.
Due to the projection formula, this means that the Poincar\'e
series of the set of divisorial valuations defines the minimal
resolution of each valuation.

Let us describe a divisorial valuation (or the corresponding
minimal resolution) by  the numbers $\seq m0{g+1}$ where, for
$i=0,1,\ldots,g+1$, $m_i:=m_{\sigma_i}$, $m_{g+1}=m_{\tau_g}$ if
$c=0$. In the last case the number $m_{g+1}$ is determined by the
Poincar\'e series of the filtration by the equation $m_{g+1} =
e_{g-1} m_{\sigma_g}$, where $e_{g-1}=\gcd(\seq{m}0{g-1})$ (in
this case the factor $(1-t^{m_{\tau_g}})$ itself does not appear
in the equation~(\ref{eq2}) for the Poincar\'e series of the
valuation).

The topological type, or equivalently the dual graph of the
minimal resolution, of a curve singularity is defined by the
topological type of each branch plus the intersection
multiplicities of pairs of branches (see \cite{zariski},
\cite{brieskorn}). This implies that the dual graph of the
minimal resolution of a set of divisorial valuations is
determined by the dual graph of the minimal resolution for each
divisor plus the intersection multiplicities of curvettes
corresponding to pairs of divisors.

Therefore it remains to show that the Poincar\'e series
$P_{\{v_i\}}({\seq t1r})$ determines these intersection
multiplicities. The projection formula permits to prove this for
two valuations. Moreover, the discussion above shows that we can
assume the dual graph of the minimal resolution of each divisor
known. Let these two divisors be described by the numbers
$\seq{m}0{g+1}$ and $\seq{m'}0{g'+1}$ respectively
($m_{\sigma_{g+1}}=m_{\sigma_{g+1}\sigma_{g+1}}$,
$m'_{\sigma_{g'+1}}=m_{\sigma'_{g'+1}\sigma'_{g'+1}}$ where the
vertices $\sigma_{g+1}$ and $\sigma'_{g'+1}$ of the dual graph
correspond to the divisorial valuations under consideration; see
the explanation above). Just as in the case of one divisorial
valuation, there is no cancellation in the equation~(\ref{eqPS})
for two valuations, i.e. all factors with
$\chi(\stackrel{\bullet}{E_\sigma})\ne 0$ are present and for all
of them the exponents $\mm_\sigma$ are different. This can be
shown, e.g., in the following way. Let $s\in\Gamma$ be
the maximal vertex which is $\le$ than both $\sigma_{g+1}$ and
$\sigma'_{g'+1}$. The statement about absence of cancellation
follows from the following facts:
\newline 1) The exponents
$\mm_{\sigma}$ are strictly increasing on $\Gamma$ with respect
to the partial order on it.
\newline 2) The ratio
$m_{\sigma,1}/m_{\sigma,2}$ of the coordinates of the exponent
$\mm_\sigma$ as a function on $\sigma$ is constant (say, equal to
$q$) on the set of vertices $\sigma$ such that $\sigma\le s$. This
ratio is (strictly) greater than $q$ on the set of vertices
$\sigma$ such that $s<\sigma\le \sigma_{g+1}$ and is (strictly)
smaller than $q$ on the set of vertices $\sigma$ such that $s <
\sigma\le \sigma'_{g'+1}$. (These properties were used in
\cite{london}. For a more precise description of the behaviour of
the ratio $m_{\sigma,1}/m_{\sigma,2}$ see in the proof of
Theorem~\ref{theo2}.)

Let $\mm_\sigma$ be a maximal exponent in the right hand side of
the equation~(\ref{eqPS}) among the factors with
$\chi(\stackrel{\bullet}{E_\sigma})=1$. For the coordinates of
$\mm_{\,\sigma}= (m_{\sigma 1}, m_{\sigma
2})\,(=(m_{\sigma\sigma_{g+1}}, m_{\sigma\sigma'_{g'+1}}))$ one
can distinguish three different situations (up to permutation of
the divisors).
\newline
a) $m_{\sigma 1}= m_{g+1}$.
This takes place if, for the valuation $v_1$,
the length $c$ of the last tail is positive.
Then the first valuation $v_1$ is equal
to $v_\sigma$ and the intersection multiplicity between the
corresponding curvettes is equal to $m_{\sigma 2}$.
\newline
b) $m_{\sigma 1} = m_{g}$ and $m_{\sigma 2}\neq m'_{g'}$.
This takes place if, for the valuation $v_1$, the length of the
last tail is equal to zero and the last dead end $\sigma_g$ for
this valuation does not participate in the minimal resolution of
the valuation $v_2$.
 Then
the intersection multiplicty between the curvettes is equal to
$e'_{g-1}\cdot m_{\sigma 2}$.
\newline
c) $m_{\sigma 1} = m_{g}$ and $m_{\sigma 2} = m'_{g'}$. This takes
place if the lengths of the last tails for both valuations are
equal to zero and the last dead ends for them coincide. Then the
intersection multiplicity between the curvettes is equal to the
minimum between $e'_{g-1}\cdot m_{\sigma 1}$ and $e_{g-1}\cdot
m_{\sigma 2}$.
\end{proof}

\medskip

Let $(C,0)$ be a germ of a curve in $(\C^2,0)$, and let $C=
\bigcup_{i=1}^r C_i$ be the representation of the curve
singularity $C$ as the union of its irreducible components. For
$i\in \{1,\ldots, r\}$, let $\varphi_i : (\C,0)\to (\C^2,0)$ be a
parametrization (uniformization) of the curve $C_i$, i.e.
$\mbox{Im\,} \varphi_i =C_i$ and $\varphi_i$ is an isomorphism
between $(\C,0)$ and $(C_i,0)$ outside of the origin. For $g\in
\OO_{\C^2,0}$, let $w_i(g)$ be the order of the function germ $g$
on the component $C_i$, i.e. the exponent of the leading term in
the power series decomposition of the function germ $g\circ
\varphi_i$: $g\circ \varphi_i (\tau) = a \tau^{w_i(g)} +
\mbox{\em\ terms of higher degree}$, where $a\neq 0$ (if $g\circ
\varphi_i \equiv 0$, $w_i(g)=\infty$). The order functions $\seq
w1r$ define a multi-index filtration on the ring $\OO_{\C^2,0}$.
The Poincar\'e series of this filtration $P_C(\seq t1r)$ is called
the Poincar\'e series of the plane curve singularity $(C,0)$.
In~\cite{duke} it was shown that the Poincar\'e series $P_C(\seq
t1r)$ coincides with the Alexander polynomial in several variables
of the link $C\cap S_{\varepsilon}^3\subset S_{\varepsilon}^3$,
where $S_{\varepsilon}^3$ is the sphere of small radius
$\varepsilon$ centred at the origin in $\C^2$.

Let $\pi : ({\cal X}, {\cal D})\to (\C^2,0)$ be an embedded
resolution of the plane curve singularity $(C,0)$, whose
exceptional divisor ${\cal D}$ is the union of irreducible
components $E_{\sigma}$, $\sigma\in \Gamma$. Let $m_{\sigma
\delta}$ be defined as above. For $i\in\{1,\ldots,r\}$, let
$\alpha_i\in\Gamma$ be the index of the component $E_{\alpha_i}$
of the exceptional divisor ${\cal D}$ intersecting the strict
transform of the component $C_i$ of the curve, let $m_{\sigma,i}:=
m_{\sigma \alpha_i}$, and let $\mm_{\sigma} := (m_{\sigma, 1},
\ldots, m_{\sigma, r})\in \Z_{\geq 0}^r$. Let
$\stackrel{\circ}{E_\sigma}$ be the smooth part of the component
$E_\sigma$ in the total transform $\pi^{-1}(C)$ of the curve $C$,
i.e. $E_\sigma$ itself without intersection points with other
components of $\pi^{-1}(C)$. According to~\cite{duke} one has:
\begin{equation}\label{eqPSC}
P_{C} (\seq t1r) =
\prod\limits_{\sigma\in \Gamma} (1-
\tt^{\,\mm_\sigma})^{-\chi(\stackrel{\circ}{E_\sigma})}\; .
\end{equation}

\begin{theorem}\label{theo2}
The Poincar\'e series of a plane curve singularity
defines the minimal resolution of the curve
up to combinatorial equivalence.
\end{theorem}

\begin{proof}
The equation (\ref{eqPSC}) gives the following
``projection formula": for
$i\in\{1,\ldots.r\}$ one has
$$
P_{C\setminus C_i}(\seq t1{i-1}, \seq t{i+1}r) =
\left[ \frac{P_C(\seq t1r)}
{(1-\tt^{\,\mm_{\alpha_i}})}\right]_{_{t_i=1}}\; .
$$

The minimal resolution of an irreducible plane curve singularity
({\it a branch}) has the form shown on Figure~\ref{fig1} with
$c=0$ and with an arrow at the vertex $\tau_g$ corresponding to the
strict transform of the curve. The number $g$ is equal to the
number of Puiseux pairs of the curve and the minimal resolution
could be described by the set $\seq{m}0g$, where, for $i=0,\ldots,
g$, $m_{i} := m_{\sigma_i}$ (this set coincides with the minimal
set of generators of the semigroup of values of the curve). The
Poincar\'e series of the branch is equal to
\begin{equation}\label{eq3}
P(t) = \frac{\prod\limits_{i=1}^{g} (1-t^{m_{\tau_i}})}
{\prod\limits_{i=0}^{g} (1-t^{m_{\sigma_i}}) }\,.
\end{equation}
It is clear that the Poincar\'e series determines the topological
type of the branch.

This implies that it suffices to show that from the Poincar\'e
series of the plane curve singularity $C$ one can recover the
exponent $\mm_{\,\alpha_{i_0}}$ and the semigroup of values
$S_{C_{i_0}}$ corresponding to some index $i_0\in\{1,\ldots,r\}$,
i.e. to an irreducible component $C_{i_0}$ of $C$. Notice that,
for $j\neq i_0$, the intersection multiplicity between the
irreducible components $C_{i_0}$ and $C_j$ is equal to
$m_{\alpha_{i_0}\alpha_j}$.

For  $\sigma, \sigma'\in \Gamma$, let $s(\sigma,\sigma')$ be the
index such that $[{\bf 1},\sigma] \cap [{\bf 1},\sigma'] = [{\bf
1},s(\sigma,\sigma')]$ (here $[{\bf 1},\sigma]$ is the geodesic in
the dual graph joining the first (minimal) vertex ${\bf 1}$ with
the vertex $\sigma$). Let us fix $j,k \in \{1,\ldots,r\}$. The
ratio $m_{\sigma, j}/m_{\sigma, k}$ as a function on $\sigma$ is
constant on $[{\bf 1},s(\alpha_j,\alpha_k)]$ and it is strictly
increasing on the geodesic $[s(\alpha_j,\alpha_k),\alpha_j]$
joining the vertex $s(\alpha_j,\alpha_k)$ with $\alpha_j$. For
$\sigma\notin [{\bf 1},\alpha_j]\cup [{\bf 1},\alpha_k]$, the
ratio $m_{\sigma, j}/m_{\sigma, k}$ is equal to $m_{\sigma',
j}/m_{\sigma', k}$, where $\sigma'$ is the unique vertex such that
$[{\bf 1},\sigma'] = ([{\bf 1},\alpha_j]\cup [{\bf 1},
\alpha_k])\cap [{\bf 1},\sigma]$. Just as in the proof of
Theorem~\ref{theo1}, one can see that there is no cancellation in
the equation~(\ref{eqPSC}) (for the minimal resolution).

Let $\sigma\in \Gamma$ be such that the exponent $\mm_{\sigma}$
is a maximal one among the set of exponents $\mm_{\tau}$ with
$\chi(\stackrel{\circ}{E_\tau})\neq 0$. Notice that, in contrast
to the case of divisorial valuations, a maximal
exponent is reached among components with
$\chi({\stackrel{\circ}{E_\tau}})< 0$.
If there exists $\tau\in \Gamma$ such that
$\tau > \sigma$ and $E_\tau\cap E_{\sigma}\neq \emptyset$
then $\mm_{\tau} > \mm_{\sigma}$. Thus there exists an index
$j\in\{1,\ldots,r\}$ such that $\sigma = \alpha_j$ (otherwise
the exponent $\mm_{\sigma}$ is not a maximal one). From the
comments above, it follows that the indices $j$ such that
$\sigma = \alpha_j$ are some of elements of the nonempty set
$A\subset \{1,\ldots, r\}$
consisting of the elements $j$ such that
$$
\frac{m_{\sigma, j}}{m_{\sigma, k}} \ge \frac{m_{\tau,
j}}{m_{\tau, k}} \quad \mbox{for }\forall k\in \{1,\ldots,r\}
\mbox{ and } \forall \tau\in \Gamma \mbox{ such that }
\chi(\stackrel{\circ}{E_\tau})\neq 0\;,
$$
i.e. such that the binomial $(1-\tt^{\mm_\tau})$ is present in the
equation~(\ref{eqPSC}). Let $\ell\in A$ be such that
$\alpha_\ell\neq \sigma$. Such $\ell$ exists only if the strict
transform of the component $C_{\ell}$ of the curve $C$ interstects
transversally the last dead end corresponding to a branch $C_j$
with $\alpha_j = \sigma$. In this case $\alpha_\ell <\alpha_j$ and
therefore $m_{\sigma, \ell}  < m_{\sigma, j}$ for any $j$ such
that $\alpha_j = \sigma$. Notice that, if such $\ell$ exists, it
is unique. This means that for $i_0\in A$ such that $m_{\sigma,
i_0}\ge m_{\sigma, j}$ for all $j\in A$ one has $\alpha_{i_0} =
\sigma$. Notice that such an index $i_0$ is, in general, not
unique.

Now the proof is a consequence of the following statement:

\begin{lemma}
The semigroup of values of the irreducible curve  $C_{i_0}$ is
generated by the elements of the set:
$$
\{m_{\tau, i_0} : \mbox{\rm{ for }} \tau \mbox{\rm{ such that
}}\chi(\stackrel{\circ}{E_\tau})=1 \} \cup \{m_{\sigma, j} : j\neq
i_0 \}\; .
$$
\end{lemma}

\begin{proof}
Obviously all elements of the described set belong to the
semigroup of values of the curve $C_{i_0}$. The minimal set of
generators of the semigroup of values of $C_{i_0}$ is the set $m_0
< m_1 < \ldots < m_g$ where $m_i := m_{\sigma_i}$
($i=0,1\ldots,g$) (see Figure~\ref{fig1}). If the exponent
$m_{q}$ corresponding to a dead end $\sigma_q$ of the dual graph
of $C_{i_0}$ does not appear in the Poincar\'e series of $C$, i.e.
if $\chi(\stackrel{\circ}{E_{\sigma_q}})=0$, then there exists
$k\in\{1,\ldots, r\}$ such that the strict transform $C'_k$ of the
component $C_k$ in the minimal resolution of $C_{i_0}$ intersects
$E_{\sigma_q}$. If $C'_k$ is not resolved yet, there must exist a
component $E_\tau$ with $\chi(\stackrel{\circ}{E_{\tau}})=1$,
$\tau > \sigma_q$ (produced by blowing-ups at points corresponding
to $C'_k$) such that $m_{\tau, i_0} = m_{\sigma, i_0} = m_q$.
Thus, the exponent $m_q$ does not appear among the ones in the
first set of the statement if and only if the strict transform
$C'_k$ of $C_k$ in the minimal resolution of $C_{i_0}$ is smooth
and transversal to $E_{\sigma_q}$. But in this case $m_{q}$
coincides with the intersection multiplicity between $C_{i_0}$ and
$C_k$ and so $m_q = m_{\sigma, k}$.
\end{proof}

This proves the Theorem. \end{proof}

\begin{remark}
One may consider the multi-index filtration defined by a set of
valuations corresponding to several irreducible plane curve
singularities
and several divisorial valuations.
One has a formula for the Poincar\'e series of such filtration
similar to~(\ref{eqPS}) and~(\ref{eqPSC}).
One could ask whether this
Poincar\'e series defines the dual graph of the minimal resolution
of the set of curves and divisors.
The following example shows
that this is, generally speaking, not the case. Let us consider
the divisorial
valuation and the curve defined by the minimal resolution shown
on Figure~\ref{fig2}.
\begin{figure}[h]
$$
\unitlength=1.00mm
\begin{picture}(80.00,10.00)(0,0)
\thicklines
\put(10,5){\line(1,0){23}}
\put(36,5){\circle*{0.5}}
\put(39.6,5){\circle*{0.5}}
\put(42,5){\circle*{0.5}}
\put(45,5){\circle*{0.5}}
\put(48,5){\line(1,0){23}}
\put(71,5){\vector(2,1){10}}
\thinlines
\put(10,5){\circle*{1}}
\put(20,5){\circle*{1}}
\put(30,5){\circle*{1}}
\put(51,5){\circle*{1}}
\put(61,5){\circle*{1}}
\put(61,5){\circle{2}}
\put(71,5){\circle*{1}}
\put(9,8){$1$}
\put(19,8){$2$}
\put(29,8){$3$}
\put(50,8){$p$}
\put(56,0){$p+2$}
\put(70,0){$p+1$}
\end{picture}
$$
\caption{Example.}
\label{fig2}
\end{figure}
The component $E_{p+2}$ of the divisor corresponds to a
curvette of type $A_{2p}$, i.e. $\{y^2+x^{2p+1}=0\}$, the
curve under consideration is
$\{y=0\}$. One can easily see that, for any $p$, the Poincar\'e
series of the corresponding filtration is
$$
P(u,t) = (1-t u^2)^{-1} \; ,
$$
where $u$ (respectively $t$) is the variable corresponding to the
divisorial valuation (to the curve respectively).
\end{remark}

\end{document}